\documentclass[a4, 12pt]{amsart}
\usepackage{amssymb}
\usepackage{amstext}
\usepackage{amsmath}
\usepackage{amscd}
\usepackage{latexsym}
\usepackage{amsfonts}
\usepackage[all]{xy}	%%the package to write commutative diagram

\usepackage{graphicx}

\newtheorem{Theorem}{Theorem}[section]

\newtheorem{theorem}[Theorem]{Theorem}
\newtheorem{prop}[Theorem]{Proposition}

\newtheorem{proposition}[Theorem]{Proposition}
\newtheorem{lem}[Theorem]{Lemma}

\newtheorem{lemma}[Theorem]{Lemma}

\newtheorem{cor}[Theorem]{Corollary}

\newtheorem{corollary}[Theorem]{Corollary}
\newtheorem{fact}[Theorem]{Fact}
\newtheorem{example}[Theorem]{Example}

\def\x{\underline x}
\def\y{\underline y}

\def\Hom{\operatorname{Hom}}

\def\Ext{\operatorname{Ext}}

\newcommand{\rmr}{\mathrm{r}}

\newcommand{\calA}{\mathcal{A}}

\newcommand{\calF}{\mathcal{F}}

\newcommand{\calN}{\mathcal{N}}

\newcommand{\calR}{\mathcal{R}}

\newcommand{\fka}{\mathfrak{a}}
\newcommand{\fkb}{\mathfrak{b}}

\newcommand{\fkm}{\mathfrak{m}}
\newcommand{\fkn}{\mathfrak{n}}

\newcommand{\fkp}{\mathfrak{p}}
\newcommand{\fkq}{\mathfrak{q}}

\def\H{\operatorname{H}}

\def\depth{\operatorname{depth}}

\def\Ass{\operatorname{Ass}}
\def\Assh{\operatorname{Assh}}

\def\height{\operatorname{ht}}

\def\Spec{\operatorname{Spec}}

\def\adeg{\operatorname{arith-deg}}

\def\gr{\operatorname{gr}}

\begin{document}

\title{The Chern Coefficient and Cohen-Macaulay rings }

\author{Hoang Le Truong}
\address{Institute of Mathematics, VAST, 18 Hoang Quoc Viet street, 10307, Hanoi, Viet Nam}
\email{hltruong@math.ac.vn}
\thanks{ This research is funded by Vietnam National Foundation for Science
and Technology Development (NAFOSTED) under grant number
101.04-2014.15.
\endgraf
{\it Key words and phrases:}
Chern coefficient, Irreducible submodules, the index of reducibility,   parameter ideasl,   Cohen-Macaulay.
\endgraf
{\it 2010 Mathematics Subject Classification:}
13H10, 13A30, 13B22, 13H15.}

\maketitle

\begin{abstract}
The purpose of this paper is to investigate a relationship between the index of reducibility and the Chern coefficient for primary ideals. Therefore, the main result of this paper gives a characterization of a Cohen-Macaulay ring in terms of its the index of reducibility, its Cohen-Macaulay type, and the Chern coefficient for parameter ideals. As corollaries to the
main theorem we obtained the characterizations of a Gorenstein ring in term of its Chern coefficient for parameter ideals.
\end{abstract}

\section{Introduction}

Let $(R, \frak m)$ be a Noetherian local ring with the maximal ideal $\frak m$ and $I$ an $\frak m$-primary ideal of $R$. 
One of our goals is to study the set of $I$-good filtrations of $R$. More
concretely, we will consider the set of multiplicative, decreasing filtrations of $R$ ideals,
$\calA = \{I_n\mid I_0 = R, I_{n+1} = II_n, n \gg 0\}$, integral over the $I$-adic filtration, conveniently
coded in the corresponding Rees algebra and its associated graded ring
$$\calR(\calA) =\bigoplus_{n\ge 0}I_n t^n, \gr_\calA(R) =\bigoplus_{n\ge 0} I^n/I^{n+1}.$$
We will study certain strata of these algebras. For that we will focus on the role
of the Hilbert polynomial of the Hilbert function $\ell(R/I_{n+1})$. 
%In fact, it is well known that there exists a polynomial $p_I(n)$ of degree $d=\dim R$ with rational coefficients, called the Hilbert-Samuel polynomial, such that  $\ell(R/I^{n+1})=p_I(n)$  for all large enough $n$. Then, there are integers $e_i(I)$ such that
$$P_\calA(n)=\sum\limits_{i=0}^d(-1)^ie_i(\calA)\binom{n+d-i}{d-i}.$$
These integers  $e_i(\calA)$   are called the Hilbert coefficients of $\calA$. In particularity, 
the leading coefficient  $e_0(\calA)$ is  called the multiplicity of $R$ with respect to $\calA$. On occasionally, the first Hilbert coefficient
$e_1(\calA)$ is referred to the Chern coefficient of $\calA$\cite{V2}. For Cohen-Macaulay rings, many penetrating relationships among these coefficients have been proved, beginning with Northcott’s \cite{No1}. More recently, similar questions have been
examined in general Noetherian local rings. For example, at the conference in Yokohama 2008, W. V. Vasconselos \cite{V2} (see also) posed the following conjecture:
\vskip0.2 cm
\noindent
{\bf The Vanishing Conjecture.}
Assume that $R$ is an unmixed, that is $\dim (\hat R/ P) = \dim R$ for all $P\in \Ass \hat R$, where $\hat R$ is the $\frak m$-adic completion of $R$. Then $R$ is  a Cohen-Macaulay local ring if and only if  $e_1(\frak q) \geq 0$ for some parameter ideal $\frak q$ of $R$.
\vskip0.2cm
\noindent
Recently, this conjecture has been settled by L. Ghezzi, S. Goto, J.-Y. Hong. K. Ozeki, T. T. Phuong, W. V. Vasconcelos in \cite{GGHOPV}. Moreover, S. Goto showed how one can use Hilbert coefficients of parameter ideals in order to study many classes about non-unmixed modules such as Buchsbaum modules, generalized Cohen-Macaulay modules, Vasconselos modules, sequentially Cohen-Macaulay modules and so on(see \cite{G}, \cite{CGT}). The goal of our paper is to continue this research direction. Concretely, we will give  characterizations of a Cohen-Macaulay ring in term of its Chern coefficient and its index of reducibility. Let $M$ be a finitely generated $R$-module of dimension $s$. Then we say that an $R$-submodule $N$ of $M$ is irreducible if $N$ is not written as the intersection of two larger $R$-submodules of $M$. Every $R$-submodule $N$ of $M$ can be expressed as an irredundant intersection of irreducible $R$-submodules 
of $M$ and the number of irreducible $R$-submodules appearing in such an expression depends only on $N$ and not on the 
expression. Let us call, for each $\frak m$-primary ideal $I$ of $M$, the number $\mathcal N(I;M)$ of irreducible $R$-
submodules of $M$ that appearing in an irredundant irreducible decomposition of $I M$ the index of reducibility of $M$ with 
respect to $I$.

In history, Northcott and Rees \cite{NR} proved that   every parameter ideals of a Noetherian local ring $R$ is irreducible if and only if  $R$ is Gorenstein. Addition, they showed that if every parameter ideals of a Noetherian local ring $R$ is irreducible then $R$ is Cohen-Macaulay.
After that, D. G. Northcott \cite [Theorem 3]{No} proved that for parameter ideals $\fkq$ in a Cohen-Macaulay local ring $R$,  
the index $\calN(\fkq;R)$ of reducibility is constant and independent on the choice of $\fkq$. However, the property of 
constant index of reducibility for parameter ideals does not characterize Cohen-Macaulay rings. The example of 
a non-Cohen-Macaulay local ring $R$ with $\calN(\fkq;R) = 2$  for every parameter ideal $\fkq$ was firstly given in 1964 by S. Endo and M. Narita \cite {EN}. It seems now natural to ask whether  characterize Cohen-Macaulay rings by  index of reducibility for parameter ideals. Recently, author gave characterize Cohen-Macaulay rings by  index of reducibility for parameter ideals and  the irreducible multiplicity of $M$ with respect to $I$. So the aim of our paper is to continue this research direction.

With this notation the main results of this paper are summarized into the following.
We denote by $r(R)=\ell_R(\Ext^d_R(R/\fkm,R))$ the Cohen-Macaulay type.

\begin{theorem}\label{T6.60}
Let $R$ be a Noetherian local ring with maximal ideal $\frak m$, $d=\dim R\geqslant 2$. Assume that  $R$ is unmixed, that is $\dim \hat R/\fkp=d$ for all $\fkp\in \Ass(\hat R)$. The following statements are equivalent.

\begin{itemize}
\item[$(i)$] $R$ is Cohen-Macaulay.
\item[$(ii)$] For all parameter ideals $\frak q\subseteq\fkm^2$, we
    have
$$\mathcal{N}(\frak q;R) =e_1(I)-e_1(\fkq),$$
where $I=\fkq:\fkm$.
\item[$(iii)$] For all parameter ideals $\frak q\subseteq\fkm^2$, we
    have
$$\mathcal{N}(\frak q;R) \leqslant e_1(I)-e_1(\fkq),$$
where $I=\fkq:\fkm$.

\item[$(iii)$]  For all parameter ideals $\frak q\subseteq\fkm^2$, we
    have
$$ e_1(I)-e_1(\fkq)=r(R),$$
where $I=\fkq:\fkm$.

\item[$(iv)$]  For all parameter ideals $\frak q\subseteq\fkm^2$, we
    have
$$ e_1(I)-e_1(\fkq)\le r(R),$$
where $I=\fkq:\fkm$.

\end{itemize}

\end{theorem}

From the main result, we get the following results.

\begin{cor}\label{Cmain}
Let $R$ be a Noetherian local ring with $d=\dim R\ge2$. Assume that  $R$ is unmixed, that is $\dim \hat R/\fkp=d$ for all $\fkp\in \Ass(\hat R)$. Then  for all  integers $n$  there exists  a parameter ideal $\frak q\subseteq \frak m^n$, we
    have
$$ r(R)\leqslant e_1(I;R)-e_1(\fkq;R)\leqslant \mathcal{N}(\frak q;R),$$
where $I=\fkq:\fkm$.
\end{cor}

Let us explain how this paper is organized. This paper is divided into 3 sections. In the next section, we give the characterizations of a Cohen-Macaulay ring in term of its Chern coefficient and the irreducible multiplicity for parameter ideals. The section 3 of the paper  is devoted to prove a parts of the main result and its consequences.   In the last section,
we give the characterizations of a Cohen-Macaulay ring in term of its Chern coefficient and the Cohen-Macaulay type. As corollaries, we obtained the characterizations of a Gorenstein ring in term of its Chern coefficient for parameter ideals.

\section{Irreducible multiplicity}
Throughout this paper we fix the following standard notations: Let $R$ be a Noetherian local ring with maximal ideal $\frak m$, $d=\dim R>0$,  $k=R/\fkm$ the infinite residue field.  Let $I$ be an $\frak m$-primary ideal of $R$.
%Then we are going to consider the set of all graded subalgebras $\calA$ of the integral closure of $R[It]$,
%$$R[It] \subset \calA \subset R[\overline{I}t].$$
The associated graded ring $\gr_I(R) =\bigoplus_{n\geq 0}I^n/I^{n+1}$
is a standard graded ring with $[\gr_I(R)]_0 = R/I$ Artinian. Let $M$ be a finitely generated $R$-module of dimension $s$.
Therefore the associated graded module $\gr_I(M) =\bigoplus_{n\geq 0} I^nM/I^{n+1}M$ of $I$ with respect to $M$ is a finitely generated graded $\gr_I(R)$–module. The Hilbert-Samuel function of $M$ with respect to $I$ is
$$H(n)=\ell_R(M/I^{n+1}M)=\sum\limits_{i=0}^n\ell_R(I^iM/I^{i+1}M),$$
where $\ell_R(*)$ stands for the length. For sufficiently large $n$, the Hilbert-Samuel function of $M$ with respect to $I$
$H(n)$ is of polynomial type,
$$\ell_R(M/I^{n+1}M)=\sum\limits_{i=0}^s(-1)^ie_i(I,M)\binom{n+s-i}{s-i}.$$
These integers  $e_i(I,M)$   are called the Hilbert coefficients of $M$ with respect to $I$. In the particular case, the leading coefficient  $e_0(I, M)$ is  said to be the multiplicity of $M$ with respect to $I$ and $e_1(I,M)$ is called by Vasconselos(\cite{V2}) the Chern coefficient of $I$  with respect to $M$. When $M=R$, we abbreviate $e_i(I, M)$ to $e_i(I)$ for all $i=1,\ldots,s$. A lot of results are known on the Chern coefficient in the case where $R$ is a Cohen-Macaulay ring. For example, as was proved by Northcott \cite{No1}, we always have $e_0(I) - \ell(R/I)\leq e_1(I)$. After Goto and Nishida in \cite{GNi} gave to extend Northcott’s
inequality without assuming that $R$ is a Cohen–Macaulay ring. Suppose that $I$ contains a parameter ideal $\fkq$ as a reduction. Then $e_0(I) - \ell(R/I)\leq e_1(I)- e_1(\fkq)$. The purpose of this paper is to investigate upper bound of $e_1(I)- e_1(\fkq)$ by the index of reducibility of $\fkq$, when $I=\fkq:\fkm$. Recall, we say that an $R$-submodule $N$ of $M$ is irreducible if $N$ is not written as the intersection of two larger $R$-submodules of $M$. Every $R$-submodule $N$ of $M$ can be expressed as an irredundant intersection of irreducible $R$-submodules 
of $M$ and the number of irreducible $R$-submodules appearing in such an expression depends only on $N$ and not on the 
expression. Let us call, for each $\frak m$-primary ideal $I$ of $M$, the number $\mathcal N(I;M)$ of irreducible $R$-
submodules of $M$ that appearing in an irredundant irreducible decomposition of $I M$ the index of reducibility of $M$ with 
respect to $I$. Remember that $$\mathcal{N}(I;M)=\ell_R([I M :_M \fkm]/I M).$$ 
Moreover, by Proposition 2.1 \cite{CQT}, it is well known that there exists a polynomial $p_{I,M}(n)$ of degree $s-1$ with rational coefficients such that  $$\mathcal{N}(I^n;M)=\ell_R([I^n M :_M \fkm]/I^n M)=p_{I,M}(n)$$  for all large enough $n$. Then, there are integers $f_i(I;M)$ such that
$$p_{I,M}(n)=\sum\limits_{i=0}^{s-1}(-1)^if_i(I;M)\binom{n+d-1-i}{d-1-i}.$$
%These integers  $f_i(I;M)$   are called the irreducible coefficients of $M$ with respect to $I$. 
The leading coefficient  $f_0(I; M)$ is  called the irreducible multiplicity of $M$ with respect to $I$.
When $M=R$, we abbreviate $f_0(I, M)$ to $f_0(I)$.
From above notations, in this section, the main result is stated as follows.

\medskip

\noindent

\begin{theorem}\label{T0.1}
Let $R$ be a Noetherian local ring with maximal ideal $\frak m$, $d=\dim R\geqslant 2$. Assume that  $R$ is unmixed, that is $\dim \hat R/\fkp=d$ for all $\fkp\in \Ass(\hat R)$. The following statements are equivalent.
\begin{itemize}
\item[$(i)$] $R$ is Cohen-Macaulay and $R$ is not a regular local ring.
\item[$(ii)$] For all parameter ideals $\frak q\subseteq\fkm^2$, we
    have
$$f_0(\frak q)  = e_1(I),$$
where $I=\fkq:\fkm$.
\item[$(iii)$] For all parameter ideals $\fkq \subseteq\fkm^2$, we
    have
$$f_0(\frak q)  \leqslant e_1(I),$$
where $I=\fkq:\fkm$.

  \item[$(iv)$] There exists a parameter ideal $\fkq$ such that $I^2=\fkq I$ and 
$ f_0(\frak q)  \leqslant e_1(I)$, where $I=\fkq :\fkm$.

\end{itemize}

\end{theorem}

In our proof of  Theorem \ref{T0.1}, the key of problem is described by under facts. Let $x_1,x_2,\ldots,x_s\in R(s\geq 1)$. Then $x_1,x_2,\ldots,x_s$ is called a $d$-sequence
if $$(x_1,\ldots,x_{i-1}):_Rx_j=(x_1,\ldots,x_{i-1}):_Rx_ix_j$$ 
for all $1\leqslant i\leqslant j\leqslant s$. We say that $x_1,x_2,\ldots,x_s$ forms a strong $d$-sequencein R if
$x_1^{n_1},x_2^{n_2},\ldots,x_s^{n_s}$ is a $d$-sequence in $R$ for all integers $n_i\ge1$($1\le i\le s$). See\cite{Hu} for basic but deep results on $d$-sequences. 

Now, for a while, let us assume that $R$ is a homomorphic image of a Gorenstein ring and $\dim R/\fkp=d$ for all $\fkp\in\Ass(R)$. Hence $R$ contains
a system $x_1, x_2, \ldots , x_d$ of parameters which forms a strong $d$-sequence in $R$
(see \cite[ Theorem 2.6]{Cu} or \cite [ Theorem 4.2]{Kw} for the existence of such systems of parameters).
The following result was given by Goto and Sakurai(see \cite[Theorem 2.1]{GSa}) about the existence of equality $I^2=\fkq I$, where $I=\fkq:\fkm$.

\begin{lemma}\label{L2.4}
Let $R$ be a Noetherian local ring with $d=\dim R\ge2$. Assume that  $R$ is unmixed, that is $\dim \hat R/\fkp=d$ for all $\fkp\in \Ass(\hat R)$. Then there exists a system $x_1,x_2,\ldots,x_d$ of parameters such that for all integer $n_i$ $1\le i\le d$, the equality $I^2 = \fkq I$ holds true, 
where $I = \fkq :\fkm$ and $\fkq=(x_1^{n_1},x_2^{n_2},\ldots,x_d^{n_d})$.
\end{lemma}

\medskip

Now we will apply the following result to the equality $I^2=\fkq I$ of Goto and Sakurai(see \cite[Theorem 2.1]{GSa}).

\begin{prop}\label{eminequality}
Assume that $\frak q = (x_1, ...,x_d)$ be a parameter ideal such that $I^2=\fkq I$, where $I=\fkq :\fkm$.
Then we have
$$e_1(I) - e_1(\fkq)\leqslant f_0(\frak q).$$
\end{prop}

\begin{proof}Since $I^2=\fkq I$, we have $I^{n+1}=\fkq^n I$ for all $n\ge 1$. Thus, we have
$$\ell(R/\fkq^{n+1})-\ell(R/I^{n+1})=\ell((\fkq^n(\fkq:\fkm))/\fkq^{n+1})\leqslant \ell((\fkq^{n+1}:\fkm)/\fkq^{n+1}).$$
Since $I^2=\fkq I$, we have $e_0(\fkq)=e_0(I)$. Therefore $e_1(I)-e_1(\fkq)\le f_0(\fkq)$.
\end{proof}

\begin{proof}[of Theorem \ref{T0.1}]
(i) $\Rightarrow$ (ii). Let $\fkq$ be a parameter ideal of $R$ such that $\fkq\subseteq \fkm^2$. Put $I=\fkq:\fkm$. Since $R$ is Cohen-Macaulay, we have $f_0(\fkq)=\calN(\fkq;N)=\ell(I/\fkq)$ by Theorem 1.1 in \cite{Tr}. Since $\fkq\subseteq\fkm^2$, by \cite[Theorem 2.2]{CP}, we have $I^2=\fkq I$. It follows from $R/\fkm$ is infinite that $e_1(I)=\ell(R/I)-e_0(\fkq)$ by Huneke and Ooishi(\cite{Hu}, \cite{O} or cf. \cite[Theorem 6.1]{CGT}). Since $R$ is Cohen-Macaulay, we have $f_0(\fkq)=\ell(I/\fkq)=e_1(I)$. 

(ii) $\Rightarrow$ (iii) is trivial.

(iii) $\Rightarrow$ (iv)  follows from Lemma \ref{L2.4}.

(iv) $\Rightarrow$ (i) Since $I^2=\fkq I$, by Proposition \ref{eminequality} we have $e_1(I)-e_1(\fkq)\le f_0(\fkq)$. Thus we get that
$0\le e_1(I)- f_0(\fkq)\le e_1(\fkq)$. It follows from $R$ is unmixed and the Theorem 1.1 of \cite{GGHOPV} that  $R$ is Cohen-Macaulay, as required.

\end{proof}

\begin{cor}\label{C2.3}
Let $R$ be a Noetherian local ring with $d=\dim R\ge2$. Assume that  $R$ is unmixed, that is $\dim \hat R/\fkp=d$ for all $\fkp\in \Ass(\hat R)$. Then  for all  integers $n$  there exists  a parameter ideal $\frak q\subseteq \frak m^n$, we
    have
$$ e_1(I;R)\leqslant f_0(\frak q;R) ,$$
where $I=\fkq:\fkm$.
\end{cor}
\begin{proof}
The result follows  from Theorem \ref{T0.1}.
\end{proof}

%%%%%%%%%%%%%%%%%%%%%%%%%%%%%%%%%%%%%%%%%%%%%%%%%%%
%%%%%%%%%%%%%%%%%%%%%%%%%%%%%%%%%%%%%%%%%%%%%%%%%%%
%%%%%%%%%%%%%%%%%%%%%%%%%%%%%%%%%%%%%%%%%%%%%%%%%%%

\section{Chern coefficient}

In this section, we denote by $\fkq_i$  the ideal $(x_1,\ldots,x_i)R$ for $i=1,\ldots , d$  and stipulate that  $\fkq_0$ is  the zero ideal of $R$. For a module $M$ over a ring $R$, we denote by $H^i_\fka(M)$ the $i$-th local
cohomology module of $M$ with respect to $\fka$. Then an  $R$-module $M$ is said to be a {\it generalized Cohen-Macaulay module} if 
$H^i_\fkm(M)$ are of finite length for all $i=0,1,\ldots , d-1$ (see \cite {CST}). This condition is equivalent to say that 
there exists a  parameter ideal $\fkq=(x_1,\ldots,x_d)$ of $M$ such that $\fkq H^i_\fkm(M/\fkq_j M)=0$ for all $0\le i+j < d$ (see 
\cite {T}), and such a parameter ideal was called a {\it standard parameter ideal} of $M$. It is well-known that if $M$ is a generalized Cohen-Macaulay module, then every parameter ideal of $M$ in a high enough power of the maximal ideal $\fkm$ is standard. The following  lemma can be easily derived from the basic properties of generalized Cohen-Macaulay modules (see \cite[Theorem 1.1 and Theorem 1.2]{CT} and \cite[Proposition 3.4]{Tr}).
\begin{fact}\label{L3.1}
\rm 
Let $(R, \frak m)$ be a generalized Cohen-Macualay
ring of dimension $d\geqslant 1$. Set $r_i(R) = \dim_{R/\frak m}((0):_{H^i_{\frak m}(R)} \frak m)$. 
Then the following statements hold true. 
\begin{itemize}
\item[$(1)$] There exists an integer $n$ such that for all parameter ideal $\fkq\subseteq \fkm^n$, we have $I^2=\fkq I$, where $I=\fkq:\fkm$ and
$$\calN(\frak q;R) = \sum_{i=0}^d \binom{d}{i}r_i(R),$$
\item[$(2)$] Let $\frak q = (x_1, x_2,...,x_d)$ be a standard
parameter ideal such that
$$\calN(\frak q;R) = \sum_{i=0}^d \binom{d}{i}r_i(R).$$
Then we have
\begin{itemize}
\item[$(a)$] $r_i(R/(x_1))=r_i(R)+r_{i+1}(R)$ for all $i\geqslant 0$.
\item[$(b)$] $$f_0(\frak q;R)\leqslant\begin{cases} f_0(\frak q^\prime;R^\prime) - (r_0(R) + r_1(R))&\text{ if $\dim R=2$,  }\\
f_0(\frak q^\prime;R^\prime)&\text{if  $\dim R\geqslant 3$.}
 \end{cases}
$$
\item[$(c)$]
 $$f_0(\frak q;R) \leqslant  \sum_{j=1}^{d} \binom{d-1}{j-1} r_j(R).$$
\end{itemize}
\end{itemize}
\end{fact}

Now we will apply Proposition \ref{eminequality} to generalized Cohen-Macaulay
rings. From there we get the following result.

\begin{corollary}\label{C3.2}
Let $R$ be a generalized Cohen-Macaulay
ring of dimension $d$. Then there exists an integer $n$ such that for all parameter ideals $\frak q\subseteq \frak m^n$, we have
 $$e_1(I)-e_1(\fkq)\leqslant f_0(\frak q) \leqslant \mathcal N(\frak q;R),$$
where $I=\fkq:\fkm$. Moreover, $e_1(I)-e_1(\fkq)= \mathcal N(\frak q;R)$ if and only if $R$ is Cohen-Macaulay.
\end{corollary}
\begin{proof}
Choose an integer $n$ as in Fact \ref{L3.1} 1). Let $\fkq$ be a parameter ideal such that $\fkq\subseteq\fkm^n$.  Then by Fact \ref{L3.1} 2) we have $I^2=\fkq I$ and $f_0(\fkq)\leqslant \calN(\fkq;R)$ for all parameter ideals $\fkq\subseteq\fkm^n$.  It follows  from Proposition \ref{eminequality} that 
 $$e_1(I)-e_1(\fkq)\leqslant f_0(\frak q) \leqslant \mathcal N(\frak q;R).$$

Now assume that $e_1(I)-e_1(\fkq)= \mathcal N(\frak q;R)$. Then $f_0(\frak q) = \mathcal N(\frak q;R)$. Since $\fkq\subseteq\fkm^n$, we have $f_0(\frak q;R) \leqslant  \sum_{j=1}^{d} \binom{d-1}{j-1} r_j(R)\le  \sum_{i=0}^d \binom{d}{i}r_i(R)= \calN(\frak q;R)$. Therefore $H^i_\fkm(R)=0$ for all $i\not=d$. Hence $R$ is Cohen-Macaulay.

If $R$ is Cohen-Macaulay then by Theorem \ref{T0.1} and Theorem 1.1 in \cite{Tr} we have $e_1(\fkq)=0$, $e_1(I)= f_0(\fkq)$ and $f_0(\fkq)=\calN(\fkq;R)$. Hence $e_1(I)-e_1(\fkq)= \mathcal N(\frak q;R)$, as required.
\end{proof}

%%%%%%%%%%%%%%%%%%%%%%%%%%%%%%%%%%%%%%%%%%%%%%%%%%%
%%%%%%%%%%%%%%%%%%%%%%%%%%%%%%%%%%%%%%%%%%%%%%%%%%%
%%%%%%%%%%%%%%%%%%%%%%%%%%%%%%%%%%%%%%%%%%%%%%%%%%%

Now we set $W = \H^0_\fkm(R)$. When we investigate in the case of 
$W\not =0$, we reduce $W=0$ using the next result(See \cite{CGT} and \cite{Tr}), which is well known,
plays a key role.

\begin{fact}\label{F2}
\rm Set $\overline{R}=R/W$. Then the following statements holds true.
 \begin{itemize}
\item[$(1)$] $e_1(I;\overline{R})= e_1(I;R)$ provided $d\ge 2$ (see \cite{CGT}).  
\item[$(2)$] There exists a positive integer $n_0$ such that  for all $\frak m$-primary  ideals $I\subseteq \frak m^{n_0}$, we have
 $$\mathcal{N}(I;R)
 =\mathcal{N}(I;\overline{R})+\ell((0):_R\frak m),$$
 and $$(\fkq +W):\fkm=\fkq:\fkm+W.$$
 \item[$(3)$] We have $$f_0(I;R) =
\begin{cases} f_0(\overline{I};\overline R) + \ell((0):_R\frak m)&\text{ if $\dim R=1$,  }\\
f_0(\overline{I};\overline R)&\text{if  $\dim R\geqslant 2$,}
 \end{cases}
 $$
 where $\overline{I}=(I+W)/W$.

 \end{itemize}

\end{fact}
\medskip

The next lemma shows the existence of a special superficial  element which is useful in many inductive proofs in the sequel.

\begin{fact}[\cite{GNi}]\label{F1}\rm
Suppose $\fkq$ is a reduction of $I$. Then there exists an element $x\in \fkq$ 
which is superficial for both $I$ and $\fkq$. Moreover, for such element
$x\in \fkq$, setting  $R^\prime = R/xR$, we have $e_1(I) - e_1(\fkq) = e_1(IR^\prime) - e_1(\fkq R^\prime)$ provided $d\geq 2$.

\end{fact}
\medskip

\begin{proposition}\label{P6.5}
Let $R$ be a Noetherian local ring with $d=\dim R\ge2$. Assume that  $R$ is unmixed, that is $\dim \hat R/\fkp=d$ for all $\fkp\in \Ass(\hat R)$. Assume that there exists an integer $n$ such that 
for all parameter ideals $\frak q=(x_1,x_2,\ldots,x_d)\subseteq \frak m^n$ we
    have
$$ \mathcal{N}(\frak q;R) \leqslant e_1(I)-e_1(\fkq),$$
where $I=\fkq :\fkm$.
Then $R$ is Cohen-Macaulay.
\end{proposition}

In our proof of Proposition \ref{P6.5} the following facts are the key. See \cite[ Section 3]{GN} for
the proof.  
\begin{lem}\label{finitely}
Let $R$ be a homomorphic image of a Cohen-Macaulay local
ring and assume that $\Ass(R)\subseteq \Assh(R)\cup \{\fkm\}$. Then
$$\calF=\{\fkp\in\Spec(R)\mid\height_R(\fkp)> 1={\depth}(R_\fkp)\}$$
is a finite set.
\end{lem}

\begin{proof} [of Proposition \ref{P6.5}]

We shall now show the our result by induction on the dimension of $R$. In the case $\dim R=2$, $R$ is a generalized Cohen-Macaulay ring since $R$ is unmixed. It follows from Corollary \ref{C3.2} and $ \mathcal{N}(\frak q;R)\leqslant e_1(I ; R)-e_1(\fkq ; R)$ that $R$ is Cohen-Macaulay.

Suppose that $\dim R>2$ and that our assertion holds true for $\dim R-1$. Let 
$$\mathcal F=\{\frak p \in \text{ Spec} R \mid \frak p \not=\frak m, \dim R_{\frak p}> \text{ depth} R_{\frak p}=1\}.$$
Then by Lemma \ref{finitely}, $\mathcal F$ is a finite set. We choose $x\in \frak m$ such that
$$x\not\in \bigcup\limits_{\frak p \in \text{Ass}R}\frak p\cup\bigcup\limits_{\frak p \in \mathcal F}\frak p.$$
Let $n_1>n$ be an integer such that $x^{n_1}H^1_{\frak m}(R)=0$. Put $y=x^{n_1}$. Let $A=R/(y)$. Then $\dim A=d-1$ and $\text{Ass} A\setminus \{\frak m\}= \text{Assh} A$. Therefore the unmixed component $U_A(0)$ of $0$ in $B$ has finite length, so that
$U_A(0)=\H^0_{\frak m}(A)$. We now take a system $y_2,y_3,\ldots,y_d$ of parameters of $R$-module $A$ and assume that $y_2,y_3,\ldots,y_d$ form a $d$-sequence in $A$. Then since $y$ is an $R$-regular, sequence $y=y_1,y_2,\ldots,y_d$ form $d$-sequence in $R$, whence $y_1$ is a superficial element of $R$ with respect to $\frak q=(y_1,y_2,\ldots,y_d)$. Since $\dim R\geqslant 3$, therefore for all parameter ideal $\frak q^\prime=(y_2,y_3,\ldots,y_d)\subseteq \frak m^n$ of $A$ which $y_2,y_3,\ldots,y_d$ is $d$-sequence, it follows from Fact \ref{F1} we have $\mathcal N(\frak q^\prime A;A)=\mathcal N(\frak q;R)\leqslant e_1(I;R)-e_1(\fkq;R)=e_1(I^\prime;A)-e_1(\fkq^\prime;A)$, where $I^\prime=\fkq^\prime A:_A\fkm A=(\fkq :_R\fkm)A$. Let $W=\H^0_\fkm(A)$ and $\overline{A}=A/ W$, and $\frak n=\frak mA$. By Fact \ref{F2} 2), that we can choose an integer $n_0>n$ such that 
for all parameters ideals $\frak q^\prime \subseteq \frak n^{n_0}$, we have
 $\mathcal{N}(\frak q^\prime ;A)=\mathcal{N}(\frak q^\prime ;\overline{A})+\ell(0:_A\frak n)$
and $  e_1(I^\prime;A)-e_1(\fkq^\prime;A)=e_1(\fkq^\prime \overline{A}:\fkn;\overline{A})-e_1(\fkq^\prime \overline{A};\overline{A})$.  Let $n^\prime >n_0$ be an integer such that $\frak mA\cap \H^0_\frak m(A)=0$. Let $y_2,y_3,\ldots,y_d$ be a system of parameters of $R$-module $\overline{A}$ such that $(y_2,y_3,\ldots,y_d)\subseteq \frak m^{n^\prime}$ and assume that $y_2,y_3,\ldots,y_d$ is a $d$-sequence in $\overline{A}$. Then because $(y_2,y_3,\ldots,y_d)A\cap W=0$, we have 
$y_2,y_3,\ldots,y_d$ form a $d$-sequence in $A$. 
Therefore,  we have $$\mathcal N(\fkq^\prime \overline{A};\overline{A})\leqslant e_1(\fkq^\prime \overline{A}:\fkn;\overline{A})-e_1(\fkq^\prime \overline{A};\overline{A}),$$
where $\fkq^\prime=(y_2,y_3,\ldots,y_d)$.
By hypothesis of induction on $d$, we have $\overline{A}$ is Cohen-Macaulay. Thus $\H^i_\frak m(A)=0$ for all $i\not=1,d$. It follows from the following sequence 
$$\xymatrix{0\ar[r]&R\ar[r]^{.y}&R\ar[r]&A\ar[r]&0}$$
that we have the long exact sequence
$$\xymatrix{\ldots\ar[r]&H^1_\fkm(R)\ar[r]^{.y}&H^1_\fkm(R)\ar[r]&H^{1}_\fkm(A)\ar[r]&\ldots}$$
$$\xymatrix{\ldots\ar[r]&H^i_\fkm(R)\ar[r]^{.y}&H^i_\fkm(R)\ar[r]&H^{i}_\fkm(A)\ar[r]&\ldots}.$$
Then we have $H^i_\frak m(R)=0$ for all $2\leqslant i\leqslant d-1$ and $H^1_\frak m(R)=y H^1_\frak m(R)$. Thus $H^1_\frak m(R)=0$  because $H^1_\frak m(R)$ is a finite generated $R$-module. Therefore $R$ is Cohen-Macaulay.
\end{proof}

\medskip

\begin{theorem}\label{T6.6}
Let $R$ be a Noetherian local ring with maximal ideal $\frak m$, $d=\dim R\geqslant 2$. Assume that  $R$ is unmixed, that is $\dim \hat R/\fkp=d$ for all $\fkp\in \Ass(\hat R)$. The following statements are equivalent.

\begin{itemize}
\item[$(i)$] $R$ is Cohen-Macaulay.
\item[$(ii)$] For all parameter ideals $\frak q\subseteq\fkm^2$, we
    have
$$\mathcal{N}(\frak q;R) =e_1(I)-e_1(\fkq),$$
where $I=\fkq:\fkm$.
\item[$(iii)$] For all parameter ideals $\frak q\subseteq\fkm^2$, we
    have
$$\mathcal{N}(\frak q;R) \leqslant e_1(I)-e_1(\fkq),$$
where $I=\fkq:\fkm$.

\end{itemize}

\end{theorem}

\begin{proof}
%Note that if $\dim R=1$ and $R$ is unmixed then $R$ is Cohen-Macaulay. Therefore, by Theorem \ref{C6.1}, the our result hold true for the case $\dim R=1$. Hence we shall now show the our result with $\dim R\ge 2$. 
(i) $\Rightarrow$ (ii). Let $\fkq$ be a parameter ideal of $R$ such that $\fkq\subseteq \fkm^2$. Put $I=\fkq:\fkm$. Since $R$ is Cohen-Macaulay, we have $\calN(\fkq;R)=\ell(I/\fkq)$ and $e_1(\fkq)=0$. Note that $\fkq\subseteq\fkm^2$, by \cite[Theorem 2.2]{CP}, we have $I^2=\fkq I$. It follows from $R/\fkm$ is infinite that $e_1(I)=\ell(R/I)-e_0(\fkq)$ by Huneke and Ooishi(\cite{Hu}, \cite{O} or cf. \cite[Theorem 6.1]{CGT}). Since $R$ is Cohen-Macaulay, we have $\calN(\fkq;R)=\ell(I/\fkq)=e_1(I)-e_1(\fkq)$.

 (ii) $\Rightarrow$ (iii)  and (iii) $\Rightarrow$ (iv) are trivial.
(iv) $\Rightarrow$ (i) follows from Proposition \ref{P6.5}.

\end{proof}

\begin{cor}\label{C3.7}
Let $R$ be a Noetherian local ring with $d=\dim R\ge2$. Assume that  $R$ is unmixed, that is $\dim \hat R/\fkp=d$ for all $\fkp\in \Ass(\hat R)$. Then  for all  integers $n$  there exists  a parameter ideal $\frak q\subseteq \frak m^n$, we
    have
$$ e_1(I;R)-e_1(\fkq;R)\leqslant  \mathcal{N}(\frak q;R),$$
where $I=\fkq:\fkm$.
\end{cor}
\begin{proof}
The result follows  from Theorem \ref{T6.6}.
\end{proof}

Let us note the following example of parameter ideals $\fkq$ in non-Cohen-Macaulay local rings $R$
with $\depth R = d - 1$, for which one has $e_1(I)-e_1(\fkq)=\calN(\fkq;R)$, where $I=\fkq:\fkm$. 

\begin{example}\rm{(\cite[Section 4]{GSa1})

Let $2 \leq d \leq m$ be integers. 
Let $A = k[X_1,X_2,\ldots,X_m, V,Z_1,Z_2, \ldots ,Z_d]$
be the polynomial ring with $m + d + 1$ indeterminates over a field $k$ and let
$$\fkb = (X_i \mid 1 \leq i \leq m -1)^2 + (X_2^m) + (X_iV \mid 1 \leq i \leq m) + (V^2 -
\sum\limits_{i=1}^dX_iZ_i).$$
We put $C= A /\fkb$. 
Let $M = C_+ = (x_1, x_2, \ldots, x_m) + (v) + (a_1, a_2,\ldots , a_d)$ be the
graded maximal ideal in $C$, where $x_i, v$, and $a_j$ denote the images $X_i, V$, and $Z_j$ in $C$,
respectively. Then $C$ is a $d$-dimensional graded non-Cohen-Macaulay ring with $\depth C = d - 1$
and $\ell(H^{d-1}_\fkm(C) = 1$ (\cite[Theorem 4.5]{GSa1}). We put $\fkq = (a_1, a_2,\ldots , a_d)$. Then $M^2 = \fkq M$,
whence $\fkq$ is a reduction of $M$ and $a_1, a_2, \ldots, a_d$ is a homogeneous system of parameters
for the graded ring $C$. Let $J = \fkq : M$. We then have $J^3 = \fkq J^2$ and $\ell_C(J^2/\fkq J) = 1$ (\cite[Proposition 4.7]{GSa1}). Let $R = C_M$, $I = JR$, and $Q = \fkq R$. Then since $\ell(H^{d-1}_\fkm(R) = 1$ and $\depth R=d-1$ we have
$$e_i(\frak q;R)=\begin{cases} 2m&\text{if $i=0$,  }\\
-1&\text{if  $i=1$,}\\
0&\text{if  $2\leq i\leq d$.}\\
 \end{cases}
$$
%and
%$$\ell(R/\fkq^{n+1})=2m\binom{n+d}{d}-(-1)\binom{n+d-1}{d-1}.$$
Moreover, we have $$\ell(R/I^{n+1})=2m\binom{n+d}{d}-(m-2)\binom{n+d-1}{d-1},$$
and so that $e_1(I)=m-2$ and $\calN(\fkq;R)=\ell(I/\fkq)=\ell(R/\fkq)-\ell(R/I)=2m+1-(m+2)=m-1$. Therefore we have
$$e_1(I)-e_1(\fkq)=\calN(\fkq;R),$$
as required.}
\end{example}

%%%%%%%%%%%%%%%%%%%%%%%%%%%%%%%%%%%%%%%%%%%%%%%%%%%
%%%%%%%%%%%%%%%%%%%%%%%%%%%%%%%%%%%%%%%%%%%%%%%%%%%
%%%%%%%%%%%%%%%%%%%%%%%%%%%%%%%%%%%%%%%%%%%%%%%%%%%

\section{The Cohen-Macaulay type}
In this section, we give the characterizations of a Cohen-Macaulay ring in term of its Chern coefficient and the Cohen-Macaulay type. As corollaries, we obtained the characterizations of a Gorenstein ring in term of its Chern coefficient. In order to give the proof of the main theorem, we begin the following result.   
\begin{lemma}\label{L4.0}
Let $R$ be a Noetherian local ring with $\dim R=1$. Assume that  $\frak q=(x)$ be a standard parameter ideal of $R$ such that $I^2=\fkq I$, where $I=\fkq:\fkm$ and
 $$\calN(\fkq;R) = r_1(R)+r_0(R).$$ Then we have
$$e_1(I)-e_1(\fkq)= f_0(\fkq)- r_0(R),$$
where $r_0(R)=\ell((0):_R\fkm)$.
\end{lemma}
\begin{proof}
First, we shall show that $\frak q^{n+1} :_R \frak m = \frak q^n (\frak q :_R \frak m) + ((0) :_R \frak m)$ for all $n\ge 0$. Indeed, the case $n = 0$ is trivial so we can assume that $n \ge 1$. Let $a \in (x^{n+1}) : \frak m$. Since $(x^{n+1}): \frak m \subseteq (x^{n+1}): x = (x^n) + H^0_{\frak m}(R)$, we have $a = x^nb + c$ for some $b \in R$ and $c \in H^0_{\frak m}(R)$. Since $\frak ma \subseteq (x^{n+1})$ and $\frak m x^nb \subseteq (x^n)$ we have $\frak mc \subseteq (x) \cap H^0_{\frak m}(R) = 0$. Thus $c \in (0):_R \frak m$. Therefore $x^n \frak m b  = \frak m a \subseteq (x^{n+1})$. Hence $\frak mb \subseteq (x) + H^0_{\frak m}(R)$. Since $(x)\cap H^0_\frak m(R)=0$, we have the following exact sequence
 $$0\to H^0_\frak m(R)\to R/(x)\to R/((x)+H^0_\frak m(R))\to 0.$$
It follows from $\calN(\fkq;R) = r_1(R)+r_0(R)$  that the sequence
$$0\to (0):_R\frak m\to ((x):\frak m)/(x)\to (((x)+H^0_\frak m(R)):\frak m)/((x)+H^0_\frak m(R))\to 0$$
is exact. Therefore $ b \in ((x) + H^0_{\frak m}(R)):_R \frak m = ((x):_R \frak m) + H^0_{\frak m}(R))$. Thus $b = d + e$ with some $d \in (x):_R \frak m$ and $e \in H^0_{\frak m}(R)$. In conclusion $a = x^{n}(d+e) + c = x^n d + c \in x^n((x):_R \frak m) + (0):_R \frak m$. Hence we have $(x^{n+1}) :_R \frak m \subseteq x^n((x):_R \frak m) + (0):_R \frak m$ as desired.\\
Since $I^2=\fkq I$, we have $I^{n+1}=\fkq^n I$ for all $n\ge 1$. Since $\fkq\cap H^0_\frak m(R)=0$ and $\frak q^{n+1} :_R \frak m = \frak q^n (\frak q :_R \frak m) + ((0) :_R \frak m)$, we have the following exact sequence
 $$0\to (0):_R \fkm\to \frak q^{n+1} :_R \frak m/\fkq^{n+1}\to (\frak q^n (\frak q :_R \frak m))/\fkq^{n+1}\to 0.$$
Thus, we have
$$\ell(R/\fkq^{n+1})-\ell(R/I^{n+1})=\ell((\fkq^n(\fkq:\fkm))/\fkq^{n+1})= \ell((\fkq^{n+1}:\fkm)/\fkq^{n+1})-\ell((0):_R\fkm).$$
Since $I^2=\fkq I$, we have $e_0(\fkq)=e_0(I)$. Therefore $e_1(I)-e_1(\fkq)= f_0(\fkq)- r_0(R)$.

\end{proof}

\begin{lemma}\label{L4.1}
Let $R$ be a generalized Cohen-Macaulay
ring of dimension $d\geqslant 2$. Assume that  $\frak q=(x_1,x_2,\ldots,x_d)$ be a standard parameter ideal of $R$ such that
$I^2=\fkq I$, where $I=\fkq:\fkm$ and
$$\calN(\frak q;R) = \sum_{i=0}^d \binom{d}{i}r_i(R).$$
 Then we have
$$e_1(I)-e_1(\fkq)= f_0(\fkq)=\sum_{j=1}^{d} \binom{d-1}{j-1} r_j(R).$$

\end{lemma}
\begin{proof}
Let $R^\prime =R/(x_1)$, $\frak q^\prime=\frak q/(x_1)$, $I^\prime=I/(x_1)$ and $\fkm^\prime=\fkm/(x_1)$. We shall now show the our result by induction on the dimension of $R$. In the case $\dim R=2$. Since $\dim R^\prime=1$ and $\fkq^\prime$ is a parameter ideal of $R^\prime$, we have
$$e_1(I^\prime)-e_1(\fkq^\prime)= f_0(\fkq^\prime)- r_0(R^\prime).$$
Because of Fact \ref{F2} (3), we have $f_0(\fkq^\prime)=r_0(R^\prime)+r_1(R^\prime)$.
It follows from Fact \ref{L3.1} (2), Fact \ref{F1} and Proposition \ref{eminequality} that we have
$$e_1(I^\prime)-e_1(\fkq^\prime)=e_1(I)-e_1(\fkq)\le f_0(\fkq)\leqslant f_0(\fkq^\prime)- (r_0(R) + r_1(R))=f_0(\fkq^\prime)- r_0(R^\prime)=r_1(R^\prime).$$ 
Hence we have $e_1(I)-e_1(\fkq)= f_0(\fkq)=r_1(R^\prime)=r_1(R)+r_2(R)$.
  
Suppose that $\dim R>2$ and  our assertion holds true for $\dim R-1$.  
By Fact \ref{F1} (2), we have
$f_0(\frak q;R)\le f_0(\frak q^\prime;R^\prime)$.
By the inductive hypothesis and Proposition \ref{eminequality}, we have
$$f_0(\fkq^\prime)=e_1(I^\prime)-e_1(\fkq^\prime)=e_1(I)-e_1(\fkq)\le f_0(\fkq)\leqslant f_0(\fkq^\prime).$$ 
and $f_0(\fkq^\prime)=\sum_{j=1}^{d-1} \binom{d-2}{j-1} r_j(R^\prime)=\sum_{j=1}^{d} \binom{d-1}{j-1} r_j(R)$
Hence we get
$$e_1(I)-e_1(\fkq)= f_0(\fkq)=\sum_{j=1}^{d} \binom{d-1}{j-1} r_j(R),$$
as required.

\end{proof}

\begin{corollary}\label{C4.2}
Let $R$ be a generalized Cohen-Macaulay
ring of dimension $d\geqslant 2$. Then there exists an integer $n$ such that for all parameter ideals $\frak q\subseteq \frak m^n$, we have
 $$e_1(I)-e_1(\fkq)= f_0(\frak q) =\sum_{j=1}^{d} \binom{d-1}{j-1} r_j(R),$$
where $I=\fkq:\fkm$.
\end{corollary}
\begin{proof}
The result follows  from Theorem \ref{L4.1} and Fact \ref{L3.1}.
\end{proof}

Let $\fkq = (x_1, x_2,\ldots,x_d)$ be a parameter ideal in $R$ and let $M$ be an $R$-module. For
each integer $n\geq 1$ we denote by $\x^n$ the sequence $x^n_1, x^n_2,\ldots,x^n_d$. Let $K^{\bullet}(x^n)$ be the
Koszul complex of $R$ generated by the sequence $\x^n$ and let
$H^{\bullet}(\x^n;M) = H^{\bullet}(\Hom_R(K^{\bullet}(\x^n),M))$
be the Koszul cohomology module of $M$. Then for every $p\in\Bbb Z$ the family $\{H^p(\x^n;M)\}_{n\ge 1}$
naturally forms an inductive system of $R$-modules, whose limit
$$H^p_\fkq=\lim\limits_{n\to\infty} H^p(\x^n;M)$$
is isomorphic to the local cohomology module
$$H^p_\fkm(M)=\lim\limits_{n\to\infty} \Ext_R^p(R/\fkm^n,M)$$
For each $n\geq 1$ and $p \in\Bbb Z$ let $\phi^{p,n}_{\x,M}:H^p(\x^n;M)\to H^p_\fkm(M)$ denote the canonical
homomorphism into the limit. With this notation we have the following.

\begin{lem}[\cite{GSa1} Lemma 3.12]\label{sur}
Let $R$ be a Noetherian local ring with the maximal ideal $\fkm$ and $
\dim R=d \ge1$. Let $M$ be a finitely generated $R$-module. Then there exists an integer $\ell$ 
such that for all systems of parameters $\x=x_1,\ldots,x_d$  for $R$ contained in $\fkm^\ell$ and for all $p\in \Bbb Z$, the canonical homomorphisms
$$\phi^{p,1}_{\x,M}:H^p(\x,M)\to H^p_\fkm(M)$$
into the inductive limit are surjective on the socles.
\end{lem}

With this notation we have the following.

\begin{lem}[\cite{GS1}, Lemma 1.7]\label{split}  Let  $M$  be a finitely generated $R$-module and $x$ be an $M$-regular element and $\x=x_1,\ldots,x_r$ be a system of elements in $R$ with $x_1= x$. Then there exists a splitting  exact sequence for each $p \in\Bbb Z$,
$$0\to H^p(\x;M)\to H^p(\x;M/xM)\to H^{p+1}(\x;M)\to0.$$ 
\end{lem}

 Let $L$ be an arbitrary finitely generated $R$-module of dimension $s \ge 0$. 
We put $$\rmr_R(L)=\ell_R(\Ext^s_R(R/\fkm,L))$$ 
and call it the Cohen-Macaulay type of $L$. (Let us simply write $\rmr (R)$ for $L = R$.) We then have 
$$\calN(\fkq;L)=\rmr_R(L/\fkq L)$$
for a parameter ideal $\fkq$ of $L$. As is well known, if $L$ is a Cohen-Macaulay $R$-module,  then for every parameter ideal $\fkq$ of $L$, we have  
 $$\calN(\fkq;L)=\ell_R(\Ext^s_R(R/\fkm,L))=\ell_R((0):_{\H^s_\fkm(L)}\fkm).$$
 The following result give
the characterizations of a Cohen-Macaulay ring in term of its Chern coefficient and its Cohen-Macaulay type.
%%%%%%%%%%%%%%%%%%%%%%%%%%%%%%%%%%%%%%%%%%%%%%%%%%%%%%%%%%%%%%%%%%%%%
%%%%%%%%%%%%%%%%%%%%%%%%%%%%%%%%%%%%%%%%%%%%%%%%%%%%%%%%%%%%%%%%%%%%%

\begin{proposition}\label{P4.50}
Let $R$ be a Noetherian local ring with $d=\dim R\ge2$. Assume that  $R$ is unmixed, that is $\dim \hat R/\fkp=d$ for all $\fkp\in \Ass(\hat R)$. Assume that there exists an integer $n$ such that 
for all parameter ideals $\frak q=(x_1,x_2,\ldots,x_d)\subseteq \frak m^n$ we
    have
$$ e_1(I)-e_1(\fkq)\leq r(R),$$
where $I=\fkq :\fkm$.
Then $R$ is Cohen-Macaulay.
\end{proposition}

\begin{proof}
We shall now show the our result by induction on the dimension of $R$. In the case $\dim R=2$, $R$ is a generalized Cohen-Macaulay ring since $R$ is unmixed. By Corollary \ref{C4.2}, we have $e_1(I)-e_1(\fkq)= f_0(\frak q) =\sum_{j=1}^{d} \binom{d-1}{j-1} r_j(R)$. Since $e_1(I)-e_1(\fkq)\le r(R)=r_d(R)$, $r_j(R)=0$ for all $j\not= d$. Therefore $R$  is Cohen-Macaulay.

Suppose that $\dim R>2$ and our assertion holds true for $\dim R-1$. Let 
$$\mathcal F=\{\frak p \in \text{ Spec} R \mid \frak p \not=\frak m, \dim R_{\frak p}> \text{ depth} R_{\frak p}=1\}.$$
Then by Lemma \ref{finitely}, $\mathcal F$ is a finite set. We choose $x\in \frak m$ such that
$$x\not\in \bigcup\limits_{\frak p \in \text{Ass}R}\frak p\cup\bigcup\limits_{\frak p \in \mathcal F}\frak p.$$
By Lemma \ref{sur} there exists an integer $\ell$  
such that for all systems of parameters $\x=x_1,\ldots,x_d$  for $R$ contained in $\fkm^\ell$ and for all $p\in \Bbb Z$, the canonical homomorphisms
$$H^p(\x,R)\to H^p_\fkm(R)$$
into the inductive limit are surjective on the socles.
Let $n_1>\max\{n,\ell\}$ be an integer such that $x^{n_1}H^1_{\frak m}(R)=0$. Put $y=x^{n_1}$. Let $A=R/(y)$. Then $\dim A=d-1$ and $\text{Ass} A\setminus \{\frak m\}= \text{Assh} A$. Therefore the unmixed component $U_A(0)$ of $0$ in $B$ has finite length, so that
$U_A(0)=\H^0_{\frak m}(A)$. We now take a system $y_2,y_3,\ldots,y_d$ of parameters of $R$-module $A$ and assume that $y_2,y_3,\ldots,y_d$ form a $d$-sequence in $A$. Then since $y$ is an $R$-regular, sequence $y=y_1,y_2,\ldots,y_d$ form $d$-sequence in $R$, whence $y_1$ is a superficial element of $R$ with respect to $\frak q=(y_1,y_2,\ldots,y_d)$. Since $\dim R\geqslant 3$, therefore for all parameter ideal $\frak q^\prime=(y_2,y_3,\ldots,y_d)\subseteq \frak m^n$ of $A$ which $y_2,y_3,\ldots,y_d$ is $d$-sequence, it follows from Fact \ref{F1} we have $e_1(I;R)-e_1(\fkq;R)=e_1(I^\prime;A)-e_1(\fkq^\prime;A)$, where $I^\prime=\fkq^\prime A:_A\fkm A=(\fkq :_R\fkm)A$.

On the other hand, by Lemma \ref{sur}, we have  the canonical homomorphism
$$H^i(\y,R)\to H^i_\fkm(R)$$
into the inductive limit are surjective on the socles, for each $i\in \Bbb Z$ where $\y=y_1,y_2\ldots,y_d$. By the regularity of $y=y_1$ on $R$, it follows from the following sequence 
$$\xymatrix{0\ar[r]&R\ar[r]^{.y}&R\ar[r]&A\ar[r]&0}$$
that there are induced the diagram
$$\xymatrix{0\ar[r]&H^i(\y;R)\ar[d]\ar[r]&H^i(\y,A)\ar[r]\ar[d]&H^{i+1}(\y;R)\ar[r]\ar[d]&0\\
\ar[r]&H^i_\fkm(R)\ar[r]&H^i_\fkm(A)\ar[r]&H^{i+1}_\fkm(R)\ar[r]&}$$
commutes, for all $i\in \Bbb Z$. 
 It follows from the above commutative diagrams and Lemma \ref{split} that after apllying the functor $\Hom(k,*)$, we obtain the commutative diagram 
$$\xymatrix{\Hom(k,H^i(\y,A))\ar[r]\ar[d]&\Hom(k,H^{i+1}(\y;R))\ar[r]\ar[d]&0\\
\Hom(k,H^i_\fkm(A))\ar[r]&\Hom(k,H^{i+1}_\fkm(R))}$$
for all $i\in \Bbb Z$. Since the map $\Hom(k,H^{i+1}(\y;R))\to\Hom(k,H^{i+1}_\fkm(R))$ is surjective, so is the map $\Hom(k,H^i_\fkm(A))\to\Hom(k,H^{i+1}_\fkm(R))$. In particular, $\Hom(k,H^{d-1}_\fkm(A))\to\Hom(k,H^d_\fkm(R))$ is surjective
and so that $r(R)\le r(A)$. Hence we have
$$e_1(I^\prime;A)-e_1(\fkq^\prime;A)\le r(A).$$

 Let $W=\H^0_\fkm(A)$ and $\overline{A}=A/ W$, and $\frak n=\frak mA$. By Fact \ref{F2} 1) and 2), that we can choose an integer $n_0>n$ such that 
for all parameters ideals $\frak q^\prime \subseteq \frak n^{n_0}$, $\fkq^\prime A:_A\fkm A+W= (\fkq^\prime A+W):_A\fkm A$ and so that we have
$  e_1(I^\prime;A)-e_1(\fkq^\prime;A)=e_1(\fkq^\prime \overline{A}:\fkn;\overline{A})-e_1(\fkq^\prime \overline{A};\overline{A})$.  Let $n^\prime >n_0$ be an integer such that $\frak mA\cap \H^0_\frak m(A)=0$. Let $y_2,y_3,\ldots,y_d$ be a system of parameters of $R$-module $\overline{A}$ such that $(y_2,y_3,\ldots,y_d)\subseteq \frak m^{n^\prime}$ and assume that $y_2,y_3,\ldots,y_d$ is a $d$-sequence in $\overline{A}$. Then because $(y_2,y_3,\ldots,y_d)A\cap W=0$, we have 
$y_2,y_3,\ldots,y_d$ form a $d$-sequence in $A$. 
Therefore, since $d\ge 3$, we have $$e_1(\fkq^\prime \overline{A}:\fkn;\overline{A})-e_1(\fkq^\prime \overline{A};\overline{A})\le r(A)=r(\overline{A}),$$
where $\fkq^\prime=(y_2,y_3,\ldots,y_d)$.
By hypothesis of induction on $d$, we have $\overline{A}$ is Cohen-Macaulay. Thus $\H^i_\frak m(A)=0$ for all $i\not=1,d$. It follows from the following sequence 
$$\xymatrix{0\ar[r]&R\ar[r]^{.y}&R\ar[r]&A\ar[r]&0}$$
that we have the long exact sequence
$$\xymatrix{\ldots\ar[r]&H^1_\fkm(R)\ar[r]^{.y}&H^1_\fkm(R)\ar[r]&H^{1}_\fkm(A)\ar[r]&\ldots}$$
$$\xymatrix{\ldots\ar[r]&H^i_\fkm(R)\ar[r]^{.y}&H^i_\fkm(R)\ar[r]&H^{i}_\fkm(A)\ar[r]&\ldots}.$$
Then we have $H^i_\frak m(R)=0$ for all $2\leqslant i\leqslant d-1$ and $H^1_\frak m(R)=y H^1_\frak m(R)$. Thus $H^1_\frak m(R)=0$  because $H^1_\frak m(R)$ is a finite generated $R$-module. Therefore $R$ is Cohen-Macaulay. 
\end{proof}

\begin{theorem}\label{T4.60}
Let $R$ be a Noetherian local ring with maximal ideal $\frak m$, $d=\dim R\geqslant 2$. Assume that  $R$ is unmixed, that is $\dim \hat R/\fkp=d$ for all $\fkp\in \Ass(\hat R)$. The following statements are equivalent.

\begin{itemize}
\item[$(i)$] $R$ is Cohen-Macaulay.
\item[$(ii)$]  For all parameter ideals $\fkq\subseteq \fkm^2$, we
    have
$$ e_1(I)-e_1(\fkq)=r(R),$$
where $I=\fkq:\fkm$.

\item[$(iii)$]  For all parameter ideals $\fkq\subseteq \fkm^2$, we
    have
$$ e_1(I)-e_1(\fkq)\le r(R),$$
where $I=\fkq:\fkm$.

\end{itemize}

\end{theorem}
\begin{proof}
(i) $\Rightarrow$ (ii). Let $\fkq$ be a parameter ideal of $R$ such that $\fkq\subseteq \fkm^2$. Put $I=\fkq:\fkm$. Since $R$ is Cohen-Macaulay, we have $\calN(\fkq;R)=\ell(I/\fkq)$ and $e_1(\fkq)=0$. Note that $\fkq\subseteq\fkm^2$, by \cite[Theorem 2.2]{CP}, we have $I^2=\fkq I$. It follows from $R/\fkm$ is infinite that $e_1(I)=\ell(R/I)-e_0(\fkq)$ by Huneke and Ooishi(\cite{Hu}, \cite{O} or cf. \cite[Theorem 6.1]{CGT}). Since $R$ is Cohen-Macaulay, we have $$e_1(I)-e_1(\fkq)=\ell(I/\fkq)=\calN(\fkq;R)=r(R).$$ 
 (ii) $\Rightarrow$ (iii) are trivial.\\
 (iii) $\Rightarrow$ (i) follows from Proposition \ref{P4.50}.

\end{proof}

\begin{cor}\label{C4.7}
Let $R$ be a Noetherian local ring with $d=\dim R\ge2$. Assume that  $R$ is unmixed, that is $\dim \hat R/\fkp=d$ for all $\fkp\in \Ass(\hat R)$. Then  for all  integers $n$  there exists  a parameter ideal $\frak q\subseteq \frak m^n$, we
    have
$$ r(R)\leqslant e_1(I;R)-e_0(\fkq;R),$$
where $I=\fkq:\fkm$.
\end{cor}
\begin{proof}
The result follows  from Theorem \ref{T4.60}.
\end{proof}

\begin{theorem}\label{T4.6}
Let $R$ be a Noetherian local ring with maximal ideal $\frak m$, $d=\dim R\geqslant 2$. Assume that  $R$ is unmixed, that is $\dim \hat R/\fkp=d$ for all $\fkp\in \Ass(\hat R)$. The following statements are equivalent.

\begin{itemize}
\item[$(i)$] $R$ is Gorenstein.
\item[$(ii)$]  For all parameter ideals $\frak q$, we
    have
$$ e_1(I)-e_1(\fkq)=1,$$
where $I=\fkq:\fkm$.

\end{itemize}

\end{theorem}
\begin{proof}
 (i) $\Rightarrow$ (ii). Since $R$ is Gorenstein, $R$ is Cohen-Macaulay ring and $r(R)=1$. By Theorem \ref{T4.60}, we have
$$ e_1(I)-e_1(\fkq)=r(R)=1$$
for all parameter ideals $\fkq$, where $I=\fkq:\fkm$.

(ii) $\Rightarrow$ (i). Since $ e_1(I)-e_1(\fkq)=1$, for all parameter ideals $\fkq$, we have 
$$e_1(I)-e_1(\fkq)\le r(R),$$
 for all parameter ideals $\fkq$.
By the Proposition \ref{P4.50}, $R$ is Cohen-Macaulay. Therefore, we have
$$1=e_1(I)-e_1(\fkq)= r(R),$$
 for all parameter ideals $\fkq$. Hence $R$ is Gorenstein, as required.

\end{proof}

%%%%%%%%%%%%%%%%%%%%%%%%%%%%%%%%%%%%%%%%%%%%%%%%%%%%
%%%%%%%%%%%%%%%%%%% References %%%%%%%%%%%%%%%%%%%%

\end{document}